\documentclass{amsart}

\usepackage[english]{babel}

\usepackage{amssymb,latexsym,amsmath,amsthm,graphicx,color}

\usepackage[pdftex,bookmarks,colorlinks,breaklinks]{hyperref}

\newtheorem{thm}{Theorem}
\newtheorem{lemma}{Lemma}
\newtheorem{cor}{Corollary}

\newtheorem{fact}{Fact}

\newtheorem*{thm-others}{Theorem}

\theoremstyle{remark}
\newtheorem*{remark}{Remark}

\newcommand{\RR}{\mathbb{R}}

\newcommand{\CC}{\mathbb{C}}

\newcommand{\ol}{\overline}
\newcommand{\p}{\partial}

\title{Zeros of harmonic polynomials, critical lemniscates and caustics}
\author{Dmitry Khavinson, Seung-Yeop Lee, Andres Saez}

\begin{document}
\begin{abstract}

In this paper we sharpen significantly several known estimates on the maximal number of zeros of complex harmonic polynomials.
We also study the relation between the
curvature of critical lemniscates and its impact on geometry of caustics and the number of zeros of harmonic polynomials.
\end{abstract}

\maketitle

\tableofcontents

\section{Introduction}

We concern ourselves in this paper with complex \emph{harmonic polynomials}, i.e., polynomials which admit a decomposition 
$$h(z) = p(z) + \ol{q(z)},$$ where $p=p_n$ and $q=q_m$ are analytic polynomials of degrees $n$ and $m$ respectively. An interesting open question is, for given $n$ and $m$, to find the maximal number of solutions to the equation $h(z)=0$, i.e., extending the Fundamental Theorem of Algebra to harmonic polynomials, see \cite{harmonious} and references therein. Throughout the paper we assume $n>m$, for the case $n=m$ could give an infinite solution set.

Wilmshurst \cite{Wil,Wil98}, in his doctoral thesis, proved the following:

\begin{thm-others}[Wilmshurst]

The equation $h(z) = 0$ has at most $n^2$ solutions.

\end{thm-others}

\noindent The proof of this result readily follows from Bezout's theorem \cite{Coo}.

Seeking to improve on this bound, Wilmshurst conjectured that the equation $h(z) = 0$ has at most $3n-2 + m(m-1)$ solutions. Khavinson and Swi\c{a}tek \cite{K-S} confirmed Wilmshurst's conjecture when $m=1$ using complex dynamics, and the bound was shown to be sharp by Geyer \cite{Gey}. However, as was shown by Lee, Lerario, and Lundberg \cite{LLL}, the conjecture is not true in general, for example, when $m=n-3$.  Also see \cite{HLLM} for many more counterexamples.

Our first theorem bounds the number of roots off the coordinate axes for the harmonic polynomials with real coefficients.

\begin{thm}\label{thm:1}
For a harmonic polynomial $h(z) = p_n(z) + \ol{q_m(z)}$ with real coefficients, the equation $h(z) = 0$  has at most $n^2 - n$ solutions that satisty $({\rm Re}\,z) ({\rm Im} z)\neq 0$.
\end{thm}

Our next two theorems provide lower bounds on the maximal number of roots.

\begin{thm}\label{thm:2}

For all $n>m$, there exists a harmonic polynomial  $h(z) = p_n(z) + \ol{q_m(z)}$ with at least $3n-2$ roots.

\end{thm}

\begin{thm}\label{thm:3} For all $n>m$, there exists a harmonic polynomial  $h(z) = p_n(z) + \ol{q_m(z)}$ with at least $m^2+m+n$ roots.
\end{thm}
The above three theorems will be proved in Section \ref{sec:3proofs}.

\vspace{0.3cm}

\noindent{\it Remark.}

\begin{itemize}%[leftmargin=0.2in]
\item[(i)] The reason why in Theorem \ref{thm:1} we only consider roots off the coordinate axes is the following.  Consider $p(z)=z^n + (z-1)^n$ and $q(z)=z^n - (z-1)^n$.  Then $h(z)=p(z)+\overline{q(z)}$ has $n^2$ number of roots including the root at $0$ with the multiplicity $n$.  In fact, this is the polynomial that Wilmshurst used (with a slight perturbation to split the multiple root at the origin) to show that the maximal bound $n^2$ is sharp.  This example shows that Theorem  \ref{thm:1} is sharp.   Theorem \ref{thm:1} also yields that the harmonic polynomial with real coefficients and with the maximal ($n^2$) number of roots should have at least $n$ roots on the coordinate axes as in the above example.   
\item[(ii)] Theorem \ref{thm:2} and Theorem \ref{thm:3} complement each other.  Theorem \ref{thm:2} is stronger than Theorem \ref{thm:3} when $m^2+m+2<2n$ and Theorem \ref{thm:3} is stronger than Theorem \ref{thm:2} when $m^2+m+2>2n$.
Also, Theorem \ref{thm:2} is not at all trivial since the argument principle for harmonic functions \cite{DHL,SS02,ST20} only yields that $h$ has at least $n$ zeroes (see Fact \ref{fact:n} in Section \ref{sec:proof1}).

\item[(iii)] Note that, compared to Wilmshurst's conjecture, Theorem \ref{thm:3} undercounts the number of roots by $2(n-m-1)$.  See Section \ref{sec:further} for the in-depth discussion.
\end{itemize}

Theorem \ref{thm:3} allows the following important corollary.
\begin{cor}\label{cor-Znm}
Let $Z_{n,m}$ denote the maximal possible number of zeros of $h=p_n+\overline{q_m}$.  Then, for any fixed integer $a\geq 1$, we have
$$\limsup_{n\to\infty}\frac{Z_{n,n-a}}{n^2}=1. $$
More generally, if $m= \alpha n + o(n)$ with $0\leq\alpha\leq 1$, we have
$$\limsup_{n\to\infty}\frac{Z_{n,m}}{n^2}\geq\alpha^2. $$
\end{cor}
\begin{proof} From Theorem \ref{thm:3} we have $Z_{n,m}\geq n^2+2(1-a) n + a(a-1)$ for the first case, and $Z_{n,m}\geq \alpha^2 n^2+o(n^2)$ for the second case.
\end{proof}

This corollary yields that the maximal number of roots is asymptotically given by Wilmshurst's theorem.
This answers a question posed by the first author more than a decade ago. Also this corollary complements the estimates on the expected number of zeros of Gaussian random harmonic polynomials obtained by Li and Wei in \cite{Li-Wei} and, more recently, by Lerario and Lundberg in \cite{LL}.  For example, for $m=\alpha n$, $\alpha <1$, the expected number of zeros is $\sim n$ (\cite{Li-Wei}) for the Gaussian harmonic polynomials and $\sim c_\alpha n^{3/2}$ \cite{LL} for the truncated Gaussian harmonic polynomials. Yet, Corollary \ref{cor-Znm} yields that among all harmonic polynomials the maximal number $\sim \alpha n^2$ of zeros occurs with positive probability, thus expanding further the results in \cite{Ble} for $m=1$.

To state our last theorem, we have to introduce the set
\begin{equation}\label{eq:Omega}
 \Omega = \{z : \lvert f(z)\rvert <1 \},
 \end{equation}
 where
 $$f(z)=\frac{p'_n(z)}{q_m'(z)}.$$
Recall that the mapping $z\mapsto h(z)$ is {\em sense-reversing} precisely on $\Omega$, i.e. the Jacobian of the map $h$ is negative on $\Omega$.  The boundary $\partial\Omega$ is the {\em lemniscate} $\{z:|f(z)|=1\}$.

Each connected component of $\Omega$ must contain at least one critical point of $p_n$. Indeed, if there is a connected component {\em without} a critical point, by applying the maximum modulus principle to $f(z)$ and $1/f(z)$ with $|f(z)|=1$ on the boundary of that component, we have that $f$ is a unimodular constant, a contradiction.

 This implies that there are at most $\deg p_n'=n-1$ connected components of $\Omega$.

For $m=1$ (when Wilmshurst's conjecture was proven to hold \cite{K-S}), Wilmshurst guessed (\cite{Wil}, p.73) that the following might be true: ``{\em In each component of $\Omega$ where $h(z)=p_n(z)+\overline z$ is sense reversing, the behavior of $h$ will be essentially determined by the $\overline z$ term so there will only be one zero of $h$}''.
From that the maximal number of roots (i.e., $3n-2$ for $m=1$) may be obtained by the argument principle when each connected component of $\Omega$ contains a zero.

The statement above is true
when the component of $\Omega$ is {\em convex} according to the following result (see \cite{Wil}, p.75).

\begin{thm-others}[Sheil-Small]
 If $g(z)$ is an analytic function in a {\em convex} domain $D$ and $|g'(z)|<1$ in $D$, then $\overline z+g(z)$ is injective on $D$.
\end{thm-others}

Note that the theorem relates the geometry of critical lemniscates with the number of zeros, because $D$ can have at most one zero of the function $\overline z+g(z)$ if the latter is injective.
In a {\em non-convex} component of $\Omega$, it is possible to have more than one zero of $h(z)$.  In \cite{Wil} an example is given where a non-convex component of $\Omega$ contains {\em two} critical points of $p_n$ and {\em two} zeros of $h(z)=p_n(z)+\overline z$.   This is not surprising because each component of $\Omega$ can have a zero of $h$ and, therefore, one can have two zeros in a component by merging two components into one.  The resulting component then has two critical points of $p_n$.      It was however not clear whether a connected component containing a single critical point of $p_n$ could possibly have more than one zero of $h$.

Here we show that it {\em is indeed} possible and, also,  present a necessary and sufficient condition for having more than one zero of $h$ in a connected component of $\Omega$ that contains a {\em single} zero of $f$ (i.e.,  a single critical point of $p_n$). The theorem holds for general $m$ and $n$.

\begin{thm}\label{thm-3} Let $n>m$. Let $D$ be a connected component of $\Omega$ (defined by \eqref{eq:Omega}) containing exactly one zero of $f$.  On a smooth part of the curve $q_m(\partial D)$ (the image of $\partial D$ under $q_m$) let $\kappa$ be the curvature of $q_m(\partial D)$ with respect to the counterclockwise arclength parametrization of $\partial D$. Then, the following are equivalent:
\begin{itemize}
\item[i)] Let $f(z)=p_n'(z)/q_m'(z)$.
There exists $z\in\partial D$ such that $$\frac{\kappa(z)}{|f'(z)|}<-\frac{1}{2}.$$
\item[ii)] There exists $\theta\in {\mathbb R}$ and $A\in{\mathbb C}$ such that $$\widetilde p_n(z)-\overline{q_m(z)}$$ has at least two zeros in $D$, where
$\widetilde p_n(z)= e^{i\theta} p_n(z)+A$.
\end{itemize}
\end{thm}

\begin{remark}{\hspace{0.1cm}}
\begin{itemize}
\item[(a)] Note that $\kappa$ is the curvature of $\partial D$ when $q_m(z)=z$.
In this case, the theorem tells us exactly how ``non-convex'' the domain $D$ needs to be in order to have multiple zeros of $h$, improving upon Sheil-Small's theorem.

\item[(b)] In statement (ii) of Theorem \ref{thm-3}, we note that
$|\widetilde p'_n(z)|=| p'_n(z)|$.
The corresponding lemniscate, $\{z:|\widetilde p_n'(z)/q'_m(z)|=1\}$, is therefore, the same for all $\theta$ and $A$.
\end{itemize}
\end{remark}

For $(n,m)=(4,1)$, we provide an example where a component of $\Omega$ with one critical point of $p_n$ contains two zeros of $h$, see Figure \ref{fig:41concave}.  Note that the roots appear where $\Omega$ is (slightly) concave.  The example is produced based on the discussion following our final Theorem \ref{lem-no-inflec} in Section \ref{sec:nonconvex} regarding the shapes of critical lemniscates.

\begin{figure}[h]
\includegraphics[width=0.6\textwidth]{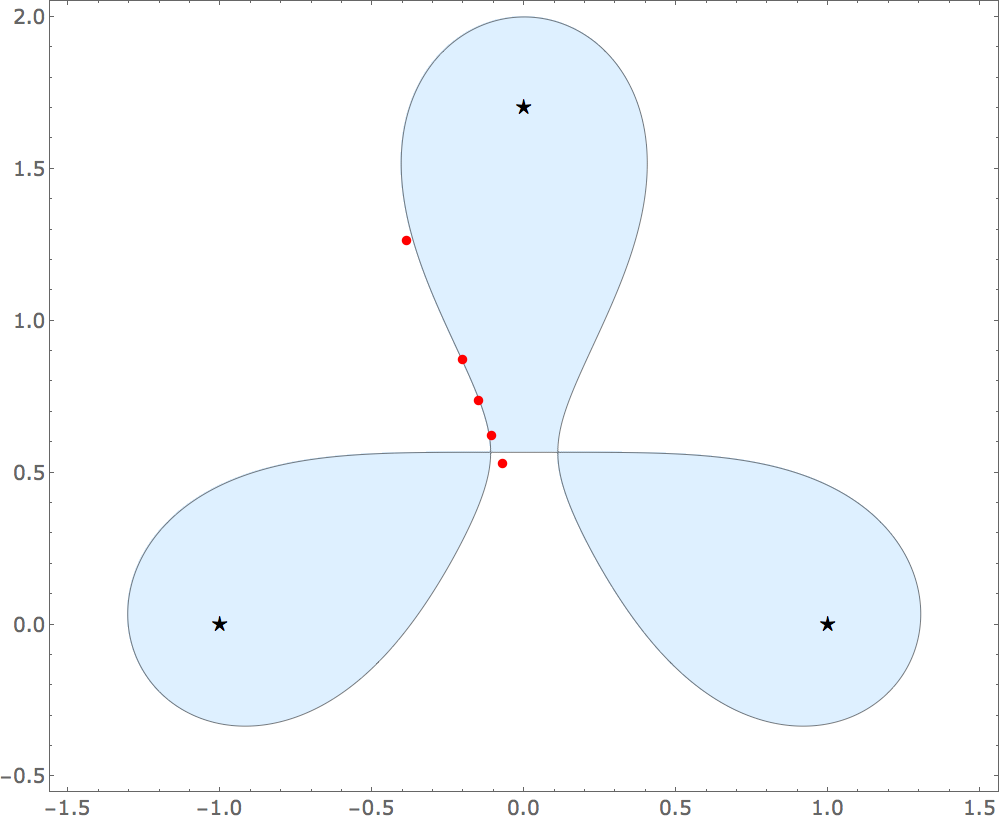}\qquad
\includegraphics[width=0.255\textwidth]{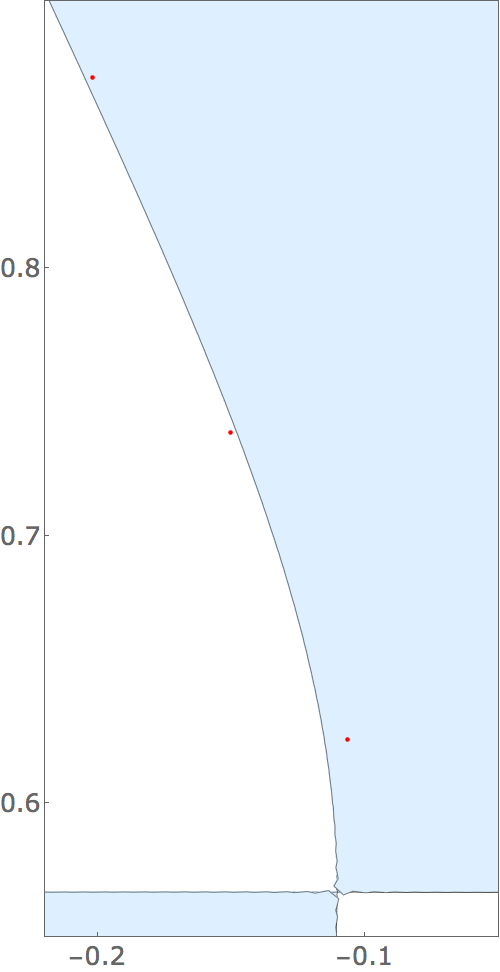}
\caption{\label{fig:41concave} Roots (dots) for  $h(z)=-(0.934124 +0.356949 i) z^*+(0.0581623 +0.156514 i) z^4+(0.354765-0.131835 i) z^3-(0.116325 +0.313028 i) z^2-(1.06429 -0.395504 i)
   z+(0.247627\, +0.020994 i)$. The stars are for the zeros of $f$ and the shaded region is $\Omega$.  The right picture is a zooming image of the top component of $\Omega$.     }
\end{figure}

\noindent{\bf Acknowledgement.} This work resulted from an REU group discussions that also included Prof. Catherine B\'en\'eteau and Brian Jackson.  The second author was supported by Simons Collaboration Grants for Mathematicians.  The first and third authors were supported by the USF Proposal Enhancement Grant No. 18326 (2015), PIs: D. Khavinson and R. Teodorescu.  We are greatly indebted to the referee for pointing out an error in the  initial version of Theorem \ref{thm:1} that was due to a misinterpretation of D. Bernstein's theorem.

\section{Proofs of Theorems \ref{thm:1}, \ref{thm:2} and \ref{thm:3}}\label{sec:3proofs}

\subsection{Proof of Theorem \ref{thm:1}}

Given a bivariate real polynomial $P$ defined by
\begin{equation*}
	P(x,y)=\sum\limits_{i=0}^n \sum\limits_{j=0}^m a_{ij}x^iy^j,\quad a_{ij}\in\RR,
\end{equation*}
the {\em Newton polygon ${\mathcal N}_P$ of $P$} is the convex hull of $N_P\subset {\mathbb R}^2$ where $N_P= \left\{(i,j) : a_{ij}\neq0\right\}$.

Given two polynomials $P$ and $Q$ with Newton polygons $\mathcal{N}_P$ and $\mathcal{N}_Q$, let $\mathcal{M}_{P,Q}$ be the {\em Minkowski sum} of $\mathcal{N}_P$ and $\mathcal{N}_Q$, defined by
$$\mathcal{M}_{P,Q} = \{(i_1+i_2,j_1+j_2) \vert (i_1,j_1) \in \mathcal{N}_P, (i_2,j_2) \in \mathcal{N}_Q\}.$$ Let $[X]$ denote the area of a set $X \subset \RR^2$. We then define the \emph{mixed area} of $P$ and $Q$ as $[\mathcal{M}_{P,Q}] - [\mathcal{N}_P] - [\mathcal{N}_Q]$.

\begin{thm-others}[D. Bernstein; cf. \cite{Ber3}, or the original articles \cite{Ber,Ber2}]

The number of solutions to the system of polynomial equations $p(x,y)=q(x,y)=0$ satisfying $x y \neq 0$ does not exceed the mixed area of $p$ and $q$.

\end{thm-others}

For any analytic polynomials $p_n(z)$ and $q_m(z)$ of degree $n,m$, we have
  $$h(x+iy)=p_n(x+iy)+\ol{q_m(x+iy)} = A(x,y)+iB(x,y),$$ where $A$ and $B$ are polynomials with real coefficients. Let $\mathcal{A}$ be the Newton polygon of $A(x,y)$ and $\mathcal{B}$ be the Newton polygon of $B(x,y)$.  As we will see below, the polynomials $p_n$ and $q_n$ having real coefficients leads to a certain structure of ${\mathcal A}$ and ${\mathcal B}$.

\begin{lemma}\label{lem:newton-poly}
Given a generic $h(z)$ with only real coefficients let ${\mathcal A}$ and ${\mathcal B}$ be defined as above.
If $n$ is even, $\mathcal{A}$ is the isosceles triangle with vertex set $\{(0,0),(0,n),(n,0)\}$ and $\mathcal{B}$ is the trapezoid with vertex set $\{(0,1),(0,n-1),(1,n-1),(n-1,1)\}$. If $n$ is odd, $\mathcal{A}$ is the trapezoid with vertex set $\{(0,0),(0,n-1),(1,n-1),(n,0)\}$ and $\mathcal{B}$ is the isosceles triangle with vertex set $\{(0,1),(0,n),(n-1,1)\}$ (see Figure 2).
\end{lemma}

\begin{figure}
\includegraphics[scale=0.2]{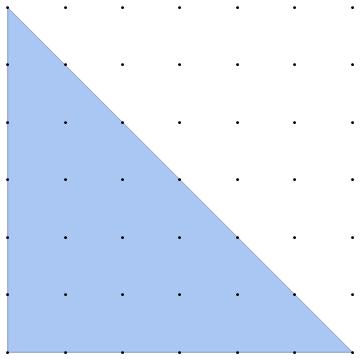} \hspace{30pt}
\includegraphics[scale=0.2]{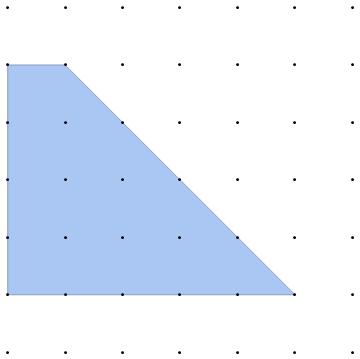}\\ \vspace{30pt}
\includegraphics[scale=0.2]{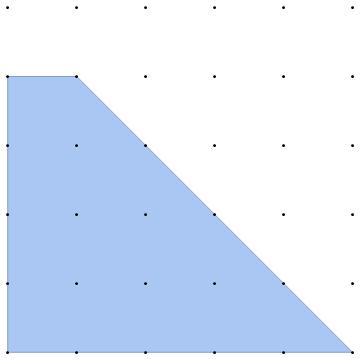}\hspace{30pt}
\includegraphics[scale=0.2]{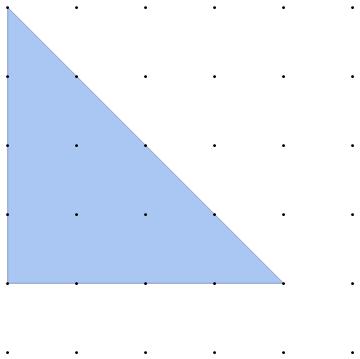}
\caption{The Newton polygons $\mathcal{A}$ (left) and $\mathcal{B}$ (right) for $n=6$ (top) and $n=5$ (bottom), respectively.}
\end{figure}

\begin{proof}
The lemma follows directly from the fact that $$A(x,y) = \sum_{ \substack{j+k \leq n,\\ k\text{ is even}} } a_{jk}x^jy^k, \qquad B(x,y) = \sum_{\substack{j+k \leq n,\\ k\text{ is odd}}} b_{jk}x^jy^k , \qquad a_{jk}, b_{jk} \in \RR.$$
The condition, $j+k\leq n$, on the summation indices follows from $\deg h=\deg p_n =n$.
The condition on $k$ to be even or odd comes from  $p_n$ and $q_m$ having real coefficients.
(Note that in the expansion of $p_n(x+iy)$ or, of $q_m(x+iy)$, the ``$i$'' comes only together with $y$ -- for a term of the form $x^j y^k$, the power of ``$i$'' is given by $k$, the power of $y$.)

We need to show that all the coefficients that correspond to the extreme points of the convex hulls are nonvanishing.
First of all, those that satisfy $j+k=n$:
\begin{equation*}
\begin{split}
&\text{$a_{0,n}$, $a_{n,0}$, $b_{1,n-1}$ and $b_{n-1,1}$ for even $n$,}\\ &\text{$a_{1,n-1}$, $a_{n,0}$, $b_{0,n}$ and $b_{n-1,1}$ for odd $n$,}
 \end{split}
 \end{equation*}
are all nonvanishing because $\deg p_n=n$.

Using $h(z)=p_n(z)+\overline{q_m(z)}$, we note that
$$h(x), \quad \frac{\partial}{\partial y}h(x+i y)\bigg|_{y=0}, \quad \frac{\partial^{n-1}}{\partial y^{n-1}}h(x+i y)\bigg|_{y=0}  $$
are all nontrivial polynomials in $x$, due to the presence of the terms, $x^n$, $x^{n-1}y$ and $x y^{n-1}$ in $h(x+iy)$ respectively.  So there exists $x_0\in\RR$ such that, when $x=x_0$, none of the above three polynomials vanishes. Defining
 $$\widetilde h(z)= h(z+ x_0),\quad x_0\in\RR,$$
 $\widetilde h$ is a real polynomial (of the same holomorphic and antiholomorphic degrees) and has the same number of zeros as $h$.  Moreover, the coefficients at the extreme points are all nonvanishing because
\begin{align*}
a_{0,0}&= \widetilde h(0) = h(x_0),
\\
b_{0,1}&=\frac{\partial}{\partial y}\widetilde h(i y)\bigg|_{y=0}=\frac{\partial}{\partial y}h(x_0+i y)\bigg|_{y=0},
\\
\heartsuit_{0,n-1}&=\frac{1}{(n-1)!}\frac{\partial^{n-1}}{\partial y^{n-1}}\widetilde h(i y)\bigg|_{y=0}=\frac{1}{(n-1)!}\frac{\partial^{n-1}}{\partial y^{n-1}}h(x_0+i y)\bigg|_{y=0},
\end{align*}
where, in the last line, the symbol $\heartsuit$ stands for $a$, when $n$ is odd, and for $b$, when $n$ is even.
\end{proof}

Now we prove Theorem \ref{thm:1}.
We divide the proof into two cases.

\noindent \textbf{Case 1.} Let $n$ be even. From Lemma \ref{lem:newton-poly}, $\mathcal{A}$ is an isosceles triangle with vertex set $\{(0,0), (0,n),(n,0)\}$ and $\mathcal{B}$ is the trapezoid with vertex set $\{(0,1),(0,n-1),(1,n-1),(n-1,1)\}$. Now, the Minkowski sum $\mathcal{M}$ of $\mathcal{A}$ and $\mathcal{B}$ is a trapezoid with vertex set $\{(0,1), (0,2n-1), (1,2n-1),(2n-1,1)\}$. The mixed area is, then,
$$[\mathcal{M}] - [\mathcal{A}] - [\mathcal{B}] = (2n^2 - 2n)-\frac{n^2}{2} - \frac{n^2-2n}{2}= n^2 - n.$$

\noindent \textbf{Case 2.} Let $n$ be odd. From the lemma, $\mathcal{A}$ is the trapezoid with vertex set $\{(0,0),(0,n-1),(1,n-1),(n,0)\}$ and $\mathcal{B}$ is the isosceles triangle with vertex set $\{(0,1),(0,n),(n-1,1)\}$. The Minkowski sum $\mathcal{M}$ of $\mathcal{A}$ and $\mathcal{B}$ is then the trapezoid with vertex set $\{(0,1),(0,2n-1),(1,2n-1),(2n-1,1)\}$. The mixed area is
$$[\mathcal{M}] - [\mathcal{A}] - [\mathcal{B}] = (2n^2 - 2n) - \frac{n^2-1}{2} - \frac{(n-1)^2}{2}=n^2-n. $$

\subsection{Proof of Theorem \ref{thm:2}}\label{sec:proof1}

For a harmonic function $h(z) = p(z) + \ol{q(z)}$, we say that $h$ is \emph{sense-preserving} at a point $z$ if the Jacobian of $h$,
 $$\det\left[\begin{array}{cc} \partial_{x} {\rm Re} \,h(x+iy ) &  \partial_{y} {\rm Re} \,h(x+iy ) \\ \partial_{x} {\rm Im} \,h(x+iy )  & \partial_{y} {\rm Im} \,h(x+iy ) \end{array} \right]_{x+iy=z} =  \lvert p'(z) \rvert ^2 - \lvert q'(z) \rvert ^2,$$
is positive, and \emph{sense-reversing} at $z$ if $\ol{h(z)}$ is sense-preserving at $z$. Otherwise, we say $h$ is \emph{singular} at $z$.  We also say that $h$ is \emph{regular} if all the zeros of $h$ are either sense-reversing, or sense-preserving.

 For an oriented, closed curve $\Gamma$ such that a continuous function $F$ does not vanish on $\Gamma$, we denote by $\Delta_{\Gamma}\arg F(z)$ the increment in the argument of $F$ along $\Gamma$.    The following is well known.

\begin{thm-others}[The argument principle for harmonic functions \cite{DHL,SS02}] Let $H$ be a harmonic function in a Jordan domain $D$ with boundary $\Gamma$.  Suppose $H$ is continuous in $\overline D$ and $H\neq 0$ on $\Gamma$.  Suppose $H$ has no singular zeros in $D$, and let $N=N_+ - N_-$, where $N_+$ and $N_-$ are the number of sense-preserving zeros and sense-reversing zeros of $H$ in $D$ respectively. Then, $\Delta_{\Gamma}\arg H(z)=2\pi N$.
\end{thm-others}

The next fact \cite{K-S} follows then by applying the argument principle  to a circle of a sufficiently large radius where $|p_n|\gg|q_m|$.

\begin{fact}\label{fact:n}
Let $h=p_n+\overline{q_m}$ be regular. Let $N_+$ be the number of sense-preserving zeros of $h$ and $N_-$ be the number of sense-reversing zeros. Then,
\begin{equation*}
	n=N_+ - N_-.
\end{equation*}
\end{fact}

An elegant proof (due to Donald Sarason) of the following lemma can be found in \cite{K-S}.
\begin{lemma}\label{lem:dense} If $p(z)$ is a polynomial of degree greater than 1, then the set of complex numbers $c$ for which $p(z) +\overline{q(z)} - c$ is regular is open and dense in $\CC$. \end{lemma}

For $m=1$, Bleher et al. \cite{Ble} proved that in the space $\CC^{n+1}$  of harmonic polynomials $p_n(z)+\overline{z}$ of degree $n\geq 2$ the set of ``simple polynomials'', i.e., regular polynomials with $k$ zeros, is a non-empty open subset of $\CC^{n+1}$ if and only if $k = n,n+2,\dots, 3n-4,3n-2$, and that the non-simple polynomials are contained in a real algebraic subset of $\CC^{n+1}$.

The following fact is noted in \cite{LSL}, Theorem 3.2.

\begin{lemma}\label{lem:regular} If the function $h(z) = p_n(z) + \ol{z}$ has $3n-2$ zeros, then $h(z)$ is regular. \end{lemma}

\begin{proof}
Assume $h$ is not regular and has exactly $3n-2$ roots.  The number of roots that are not sense-preserving is at most $n-1$ because each of those roots attracts a critical point of $p_n$ under the iteration of $z\mapsto p_n(z)$, cf. Proposition 1 in \cite{K-S}.  Therefore, the number of sense-preserving roots must be at least $2n-1$. For all non-singular roots $z_j$'s, there exists $\epsilon>0$ such that the disks, $B_{\epsilon}(z_j) = \{ z : \lvert z - z_j \rvert < \epsilon \}$, do not intersect each other, do not intersect the singular set $\{z : |p'(z)|=1\}$, and $\Delta_{\partial B_\epsilon(z_j)}\arg h =\pm 2\pi$. Defining
 $$\delta= \min_{z\in \partial B_\epsilon(z_j)} |h(z)|, $$
 for any $c\in{\mathbb C}$ with $\lvert c \rvert  < \delta$, we have $\Delta_{\partial B_\epsilon(z_j)}\arg h =\Delta_{\partial B_\epsilon(z_j)}\arg (h - c) $ and, using the argument principle for harmonic functions, the equation $h(z) - c$ has exactly one zero in each $B_{\epsilon}(z_j)$.

Suppose $z_0$ is a singular zero of $h$. Since $h(z_0) = 0$ and $h$ is continuous near $z_0$, there is an $\eta>0$ such that $B_{\eta}(z_0)$ does not intersect any $B_\epsilon(z_j)$ and
 $$\text{$\lvert h(z) \rvert < \delta$ for all $z \in B_{\eta}(z_0)$.}$$
Further, the set $B_{\eta}(z_0)$ intersects sense-preserving region,  since otherwise, $\log|p'(z)|$ would be constant over $B_\eta(z_0)$ by the maximum modulus principle.
Let $\zeta$ be a sense-preserving point in $B_{\eta}(z_0)$. We can set $c=h(\zeta)$ since $\lvert h(\zeta) \rvert < \delta$, and $h(z)-c=h(z)-h(\zeta)$ must have zeros in each  $B_{\epsilon}(z_j)$ and at $\zeta\in B_\eta(z_0)$.  Consequently, $h(z)-h(\zeta)$ has at least $2n$ sense-preserving zeros. By Lemma \ref{lem:dense}, we can choose $\zeta$ such that $h(z)-h(\zeta)$ is a regular polynomial, which contradicts the result of Khavinson and Swi\c{a}tek \cite{K-S} that the regular polynomial can have at most $2n-1$ sense-preserving roots.
\end{proof}

We now complete the proof of Theorem \ref{thm:2}.

\begin{proof}

Let $p_n(z)$ be an analytic polynomial such that the equation $p_n(z) + \ol{z}=0$ has $3n-2$ solutions. By Lemma \ref{lem:regular} the polynomial $h(z)=p_n(z)+\ol{z}$ is regular. Let $z_0$ be a zero of $h$. One can find a circle,  $\Gamma$, centered at $z_0$ with radius $\epsilon$ such that $h$ does not vanish on $\Gamma$ and $\Delta_{\Gamma}\arg h=\pm 2\pi$.  As in a standard proof of Rouch\'e's theorem, taking $\delta$ such that
$$0<\delta<  \frac{\min_{z\in\Gamma}|h(z)|}{\max_{z\in\Gamma}(|z|^m+1)},$$ we have the perturbed mapping, $z \mapsto h(z) + \delta \ol{z}^m$, that preserves the winding number of $h(\Gamma)$, that is, the perturbed mapping also vanishes in the interior of $\Gamma$. Applying the same argument to all $3n-2$ zeros of $h$, we can choose $\delta$ such that $h(z)+\delta\overline z^m$ has (at least) $3n-2$ zeros. Setting now $q_m(z)=\delta z^m+z$ completes the proof.
\end{proof}

\begin{remark}
This proof suggests that if the equation $p_n(z) + \ol{q_{m-1}(z)}=0$ has at most $k$ solutions, then there exists a harmonic polynomial $p_n(z) + \ol{q_m(z)}$ with $k$ zeros. However, a proof of this, following the above argument, would require that $p_n(z) + \ol{q_{m-1}(z)}$ be regular.
\end{remark}

\subsection{Proof of Theorem \ref{thm:3}}

Let us sketch the procedure (similar to \cite{LLL,HLLM}) that allows creating examples of harmonic polynomials with a large number of roots (cf. Figure \ref{fig:4252}).

Fix $n$ and $m<n$.
Let
\begin{equation*}
\begin{split}
	S(z)&=(z-a)^{n-1}(z+(n-1)a),
	\\  T(z)&=(z-b)^{m+1}(z^{n-m-1}+t_{n-m-2}z^{n-m-2}+\cdots+t_0).
	\end{split}
\end{equation*}
The $n-m-1$ complex parameters $t_j$'s in $T(z)$ are uniquely determined by the condition that
$S(z) -T(z)$ is a polynomial of degree $m$.  Then we choose $a\in{\mathbb C}$ and $b\in{\mathbb C}$ by hand to maximize the number of intersections between the two sets:
\begin{equation*}
	\Gamma_T=\{z\,|\, {\rm Im}\,T(z)=0\},\quad 	\Gamma_S=\{z\,|\, {\rm Re}\,S(z)=0\}.
\end{equation*}
These intersections are the roots of the equation $p_n(z)+\overline{q_m(z)}=0$ where
\begin{equation*}
	p_n(z)=S(z) + T(z),\quad q_m(z)=S(z)-T(z),
\end{equation*}
since
\begin{equation*}
	p_n(z)+\overline{q_m(z)}=2\,{\rm Re}\,S(z)+2\,i\, {\rm Im}\,T(z) .
\end{equation*}

\begin{figure}[h]
\mbox{\includegraphics[width=0.45\textwidth]{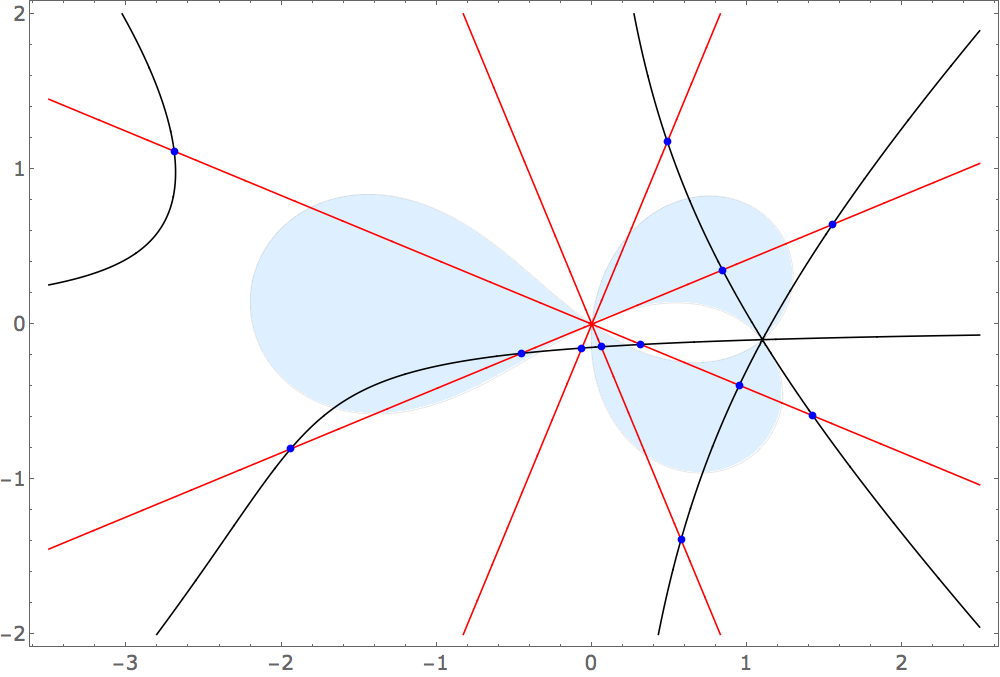}}
\hspace{5px}
\mbox{\includegraphics[width=0.45\textwidth]{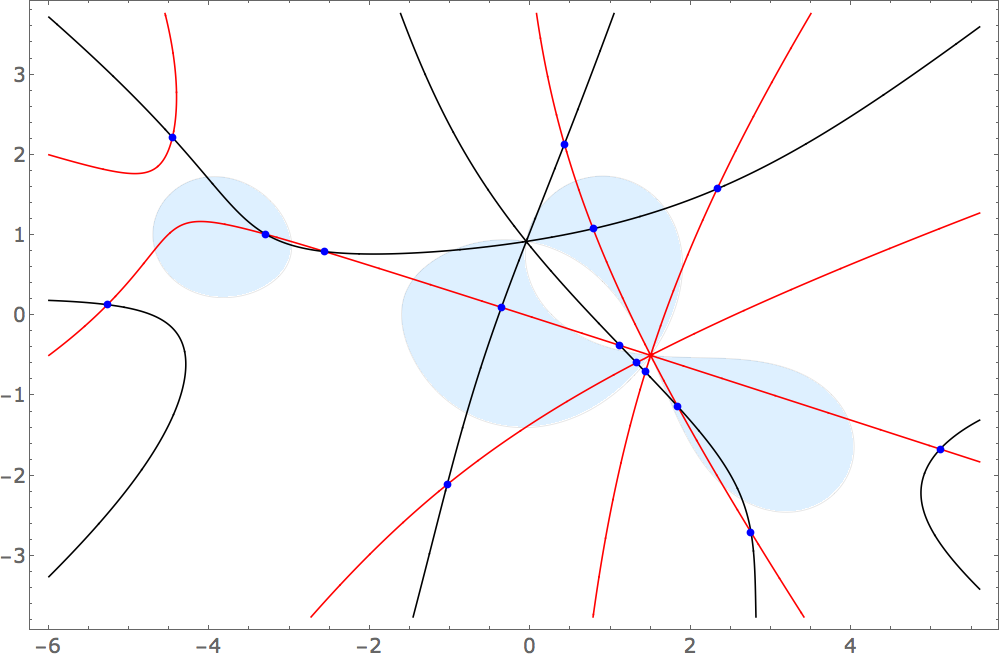}}
\caption{\label{fig:4252} Curves $\Gamma_T$ (black) and $\Gamma_S$ (red) for for $n=4, 5$ and $m=2$.  The shaded region is $\Omega$. For $n=4$ (left) we choose $a=0$ and $b=1.1-0.1i$ to produce 12 roots.   For $n=5$ (right) we choose $a=1.5-0.5i$ and $b=-0.05+0.92i$ to produce 15 roots.  }
\end{figure}

Theorem \ref{thm:3} can be stated equivalently as follows:

\begin{thm-others}[Theorem \ref{thm:3}] For $a=0$ and for a generic choice of $b\in{\mathbb C}$, the equation $p_n+\overline{q_m}=0$ defined in terms of $S$ and $T$ (as above) has at least
 $m^2+m+n$ roots.
\end{thm-others}

\begin{proof}
Choose $a=0$ and let $\Gamma_S$ be the set of rays emanating from the origin and extending to the infinity, i.e., $\{\infty\times e^{(\frac{1}{2}+k)\pi i/n}:k=0,1,\cdots,2n-1\}$.  Note that $\Gamma_T$ has $2m+2$ curved rays emanating from $b$ where every ray eventually approaches the infinity in the directions of $\{\infty\times e^{l\pi i/n}\}_{l}$ such that: i) different rays correspond to different values of $l$, and ii) $l$ is chosen in a subset, that we will denote by $N_{2m+2}$, containing $2m+2$ numbers from $\{0,1,\cdots,2n-1\}$.  Assuming that $b\notin\Gamma_S$ and that none of the rays hit any critical point of $T(z)$ except at $z=b$,  those curved rays do not intersect each other.

Let $W_l$'s ($l=0,1,\cdots,2n-1$) denote the connected components (that we will call ``sectors'') in $\CC\setminus\Gamma_S$ that contain the asymptotic direction $\infty\times e^{l\pi i/n}$.   A simple geometric consideration tells us that the number of intersections between $\Gamma_S$ and a curved ray starting from $b\in W_{l_1}$ and continuing into $W_{l_2}$ without passing the origin is at least
 $$\min(|l_2-l_1|,2n-|l_2-l_1|).$$
This means that the minimal possible number of intersections between $\Gamma_S$ and ``the $2m+2$ curved rays in $\Gamma_T$'' is
$$\sum_{l_2\in N_{2m+2} }\min(|l_2-l_1|,2n-|l_2-l_1|)\geq 0+2(1+2+\cdots+ m ) + (m+1) = (m+1)^2.$$
The remaining part of $\Gamma_T$ (i.e. that is not connected to $b$) approaches the infinity in $2n-2m-2$ different sectors among $W_l$, that are not already taken by the rays from $b$.   Assuming that there are no critical points of $T$ in $\Gamma_T$ except the one at $b$, each curve in the ``remaining part of $\Gamma_T$'' must connect two sectors from the $2n-2m-2$ sectors such that, around $\infty$, each sector is "hit" by one curve only.  Since there is at least one intersection between each curve and $\Gamma_S$, the minimal number of intersections between the ``remaining part of $\Gamma_T$'' and $\Gamma_S$ is $n-m-1$ and the minimal number of intersections between $\Gamma_T$ and $\Gamma_S$ is given by $(m+1)^2+n-m-1=m^2+m+n$.
\end{proof}

\subsection{A remark on Wilmshurst's conjecture}\label{sec:further}

\begin{figure}[h]
\begin{center}\includegraphics[width=0.7\textwidth]{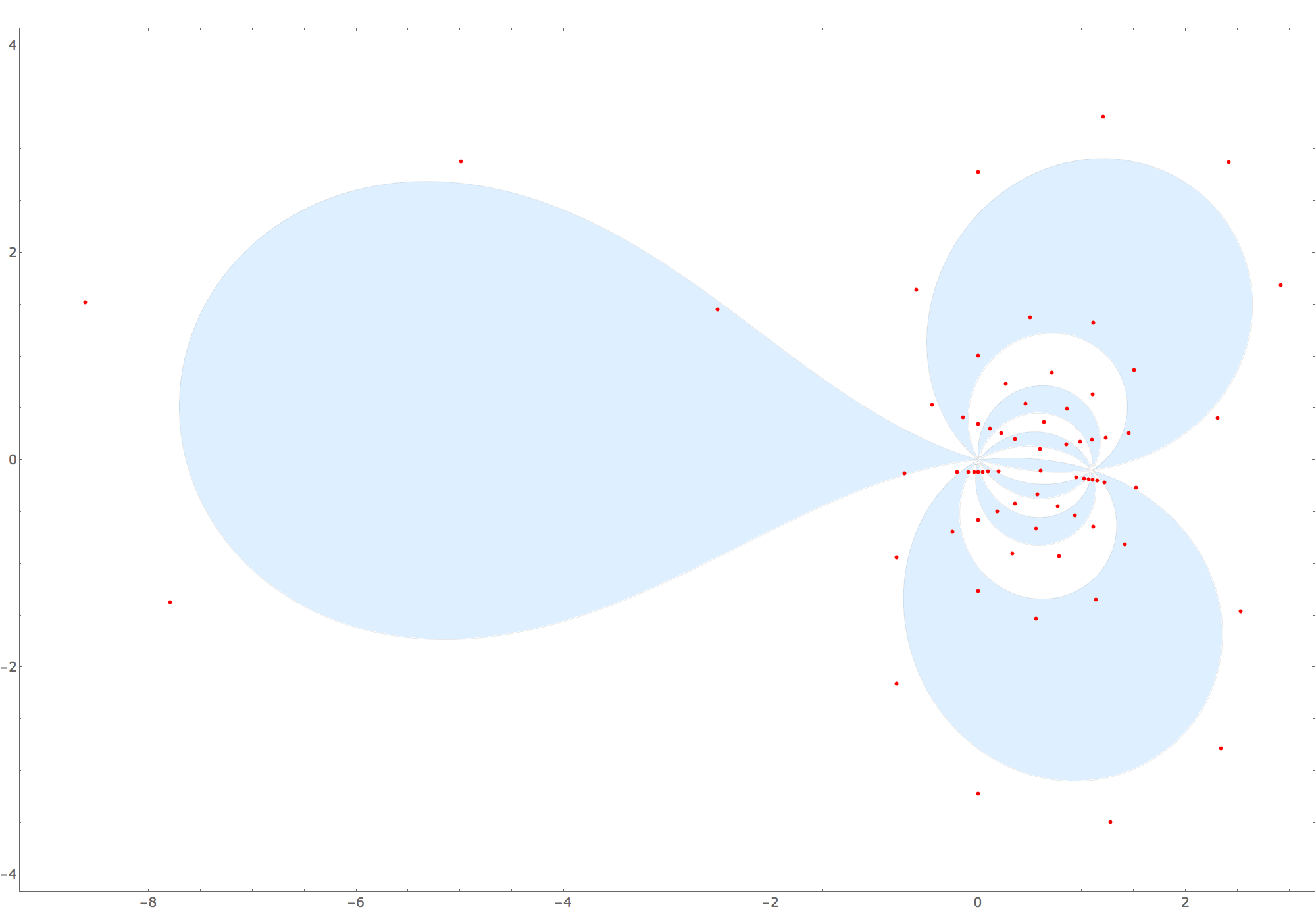}
\caption{\label{fig:97} Roots and $\Omega$ (shaded region) for $(n,m)=(9,7)$, $a=0$ and $b=1.1-0.1 i$.}
\end{center}
\end{figure}

Comparing with Wilmshurst's conjecture, Theorem \ref{thm:3} undercounts the number of roots by
\begin{equation}\label{eq:diff}
	3n-2+m(m-1) - (m^2+m+n) = 2(n-m-1).
\end{equation}
One can show that the corresponding lemniscate has $2m+2$ curves that connects $b$ and $a$ (cf. Figure \ref{fig:97}).
The numerics suggests that the $2m+1$ regions in between these curves have, respectively,
$$ m-1,m-2,\cdots,2,1,1,2,\cdots,m-1,m,$$
(counting from top to bottom in Figure \ref{fig:97}) roots.
Among these exactly $1+2+\cdots+m=m(m+1)/2$ of them are found in the sense-reversing region (with $m$ components; the shaded part in Figure \ref{fig:97}) and, therefore, the total number of roots must be at least
\begin{equation}\label{eq:mmn} 2\times \frac{m^2+m}{2}+ n =  m^2+ m + n  \end{equation}
by Fact \ref{fact:n}, giving the same number as in Theorem \ref{thm:3}.

Since there can be at most $n-1$ components in $\Omega$, there can be $n-m-1$ extra components in the sense-reversing region, and these components are not connected to the point $b$.  Wilmshurst's count is obtained {\em when each of these components has exactly one root}.  This increases the total number of {\em sense-reversing roots} by $n-m-1$ and, hence, increases the total number of roots by $2(n-m-1)$, cf. \eqref{eq:mmn}, which is precisely Wilmshurst's count, cf. \eqref{eq:diff}.

The various counterexamples studied in \cite{LLL,HLLM} indicate that, the $n-m-1$ ``extra components'' of $\Omega$ (that are not connected to $b$) can have more than one root in each component.   For example, Figure \ref{fig:97} shows that there are two roots inside the component of $\Omega$ that is not connected to $b$.

For $m=n-2$, our numerical experiment supports the structure shown in Figure \ref{fig:97}:  there are $m^2+m+n$ roots that are counted in terms of the $2m+2$ curves connecting $a$ and $b$, and the excessive zeros are twice the number of zeros found in the component of $\Omega$ that is not connected to $b$.
 Choosing $a=0$ and $b=e^{i\pi/(2n)}+\epsilon$ ($\epsilon\neq 0$ is needed so that the origin is not a root), we found that the number of ``excessive zeros'', i.e. (the total number of zeros)$-(m^2+m+n)$, increases by 4 whenever $n$ hits the numbers:
$$ 7, 15, 22, 30, 37, 45, 52, 60, 68, 75, 83, 90, 98, 105, 113, 120, 128, 136,\cdots.
$$
For example, for $n=100$, there are total 13 numbers before 100 from the list, and the total number of zeros is $4\times 13 +(m^2+m+n) = 52+9998$, exceeding Wilmshurst's count by $4\times 13 -2(n-m-1) = 52 - 2 = 50 $.
These experiments prompt us to suggest the following conjecture.

\bigskip
\noindent{\bf Conjecture.  } {\em When $m=n-2$, the maximal number of roots of $h=p_n+\overline{q_m}$ is given by
\begin{equation*}
	n^2 - \frac{3}{2} n + o(n)
\end{equation*}
as $n$ grows to $\infty$ (which is larger than Wilmshurst's count of $n^2-2n+4$).
}

\section{Geometry of caustics: Proof of Theorem \ref{thm-3}}

As before, let $\Omega$ be defined by \eqref{eq:Omega}.

\begin{lemma}\label{lem:D-univalent} Setting
 $$f(z)=p'_n(z)/q'_m(z),$$
let $D$ be a connected component of $\Omega$ with exactly $k$ (counting the multiplicities) zeros of $f$ in $D$.
Then, $f:D\to {\mathbb D}$, where ${\mathbb D}$ is the unit disk, is a branched covering of degree $k$.
 \end{lemma}
\begin{proof}
By definition, $D$ is a connected component of $f^{-1}({\mathbb D})$.   By the argument principle, for any $w\in {\mathbb D}$, the number of preimages of $w$ under $f$ inside $D$ is given by the winding number, $\frac{1}{2\pi}\Delta_{\partial D} \arg (f(z)-w)$.  Since $f(\partial D)\subset\p\mathbb D$ and cannot ``backtrack" on $\partial D$, the winding number does not depend on $w\in{\mathbb D}$ and it is $k$ at $w=0$ because $D$ contains exactly $k$ zeros of $f$.
\end{proof}

When $D$ contains exactly one critical point of $p_n$, $f:D\to{\mathbb D}$ is a univalent map.  In this case,
let $\eta:[0,2\pi)\to \partial D$ be the parametrization of $\partial D$ given by
\begin{equation*} \eta(\theta) = f^{-1}(e^{i\theta}),
\end{equation*}
where $f^{-1}$ on $\partial{\mathbb D}$ is obtained by the continuous extension of $f^{-1}:{\mathbb D}\to D$.
This parametrization of $\partial D$ is given by the harmonic measure of $D$ (normalized by the factor $2\pi$) with respect to the pole at the zero of $f$, i.e.,
\begin{equation}\label{harm-para} d\theta = d\arg f(z),\quad z\in\partial D.
\end{equation}
This viewpoint can be be generalized when there are $k$ zeros of $f$ in $D$.   In this case, the same equation \eqref{harm-para} defines the parametrization, $\eta:[0,2\pi k)\to \partial D$, by the harmonic measure of $k$-sheeted disk.

On a smooth part of the curve $\partial D$ (i.e., where $f'\neq 0$), let $v(z)$ be the tangent vector of the curve $\partial D$ at $z\in \partial D$ given by
\begin{equation}\label{eq:v}
	v(z)= \frac{d \eta(\theta)}{d\theta}=i\frac{f(z)}{f'(z)}.
\end{equation}
Let us consider the image of $\partial D$ under $h$.  For each $z\in\partial D$ with $f'(z)\neq 0$ we obtain a tangent vector of the curve $h(\partial D)$ at $h(z)$ as follows ($v$ is the tangent vector to $\partial D$, cf. \eqref{eq:v}, $|f|=1$ on $\partial D$ ):
\begin{equation}\label{eq:V}
	V(z) = \Big (v(z)\partial +\overline{v(z)}\, \overline \partial\Big) h(z) = v(z) p'_n(z) -\overline{v(z)} \,\overline{q'_m(z)},
\end{equation}
assuming this expression does not vanish. The next two lemmas deal with the geometry of the caustic that we identify now. 

\vspace{0.3cm}
\noindent{\bf Definition.} {\it The image of the critical lemniscate $h(\partial \Omega)$ is called the caustic.  }

\begin{lemma}\label{lem-caustic}
Let $D$ be a component of $\Omega$ with exactly $k$ zeros (counting multiplicity) of $f$ and $\partial D$ is parametrized by $\eta:[0,2\pi k)\to\partial D$ such that
$ d\arg f(\eta(\theta)) =d\theta$.
At $z=\eta(\theta)\in\partial D$ where $f'(z)\neq 0$ and ${\rm Im}\big(v(z) q'_m(z) \sqrt{f(z)}\big)\neq 0$, we have
\begin{equation*}
\frac{d \arg V(\eta(\theta))}{d\theta}=\frac{1}{2}.
\end{equation*}
\end{lemma}

In other words, the caustic, away from the possible singularities, has constant curvature with respect to the special parametrization defined above.

\begin{proof}
Note that, for $z\in \partial D$, the two terms in the right hand side of \eqref{eq:V} have the same modulus (i.e., $|p'_n|=|q'_m|$ on $\partial D$).
We may rewrite $V$ as
\begin{equation}\label{eq:V-cusp}
\begin{split}
	V(z)&=\sqrt{\frac{p'_n(z)}{q'_m(z)}} |q'_m(z)| \left(v(z) \sqrt{\frac{p'_n(z)}{\overline{q'_m(z)}}} -\overline{v(z)} \,\sqrt{\frac{\overline{q'_m(z)}}{p'_n(z)}}\right)
	\\&= \sqrt{\frac{p'_n(z)}{q'_m(z)}} |q'_m(z)| 2i\, {\rm Im}\left(v(z) \sqrt{\frac{p'_n(z)}{\overline{q'_m(z)}}}\right)
	\\&= 2i\,\sqrt{f(z)}  \, {\rm Im}\left(v(z) q'_m(z) \sqrt{f(z)}\right).
\end{split}
\end{equation}
(Note that the result is independent of the branch of the square root function as $\sqrt{f(z)}$ appears twice.)
If ${\rm Im}\Big(v(z) q'_m(z) \sqrt{f(z)}\Big)\neq 0$, then we have, modulo $\pi$,
\begin{equation*}
	\arg V(z)
	=\frac{1}{2}\arg f(z) + \frac{\pi}{2}=\frac{-i}{2}\log f(z) +\frac{\pi}{2},
\end{equation*}
where, in the last equality, we used again that $|f(z)|=1$ at $z\in\partial D$.
We have:
\begin{equation}\label{eq:V-sqrt-f}
	\frac{d\arg V(\eta(\theta))}{d\theta} = \frac{-i}{2}\frac{f'}{f}\frac{d\eta(\theta)}{d\theta}=\frac{1}{2},
\end{equation}
where we used  \eqref{eq:v} in the last equality.
\end{proof}

The next lemma characterizes the ``possible singularities''.

\begin{lemma}\label{lem:caustic-cusp}
 The only singularities of the curve $h(\partial D)$ are cusps.   When $z_0\in\partial D$ is not in the branch cut of $\sqrt{f}$ there is a cusp at $h(z_0)$ if and only if the mapping from $\partial D$ to ${\mathbb R}$ given by
$$z\mapsto  {\rm Im}\Big(v(z) q'_m(z) \sqrt{f(z)}\Big) \text{~~ on }\partial D,$$ changes sign across $z_0$.  When $z_0$ is in the branch cut of $\sqrt{f}$, the cusp occurs at $h(z_0)$ if and only if the mapping,
$$z\mapsto  {\rm Re}\left(\frac{\sqrt{f(z)}}{\sqrt{f(z_0)}}\right) \,{\rm Im}\Big(v(z) q'_m(z) \sqrt{f(z)}\Big) \text{~~ on }\partial D,$$
changes sign across $z_0$.
\end{lemma}

\begin{proof}

First note that away from the branch cuts of $\sqrt f$, the set
$$\partial D\cap\{z:{\rm Im}\big(v(z) q'_m(z) \sqrt{f(z)}\big)=0\}$$
is  finite. Indeed, otherwise as a level set of a harmonic function it would contain an arc.  By \eqref{eq:V}, $V=0$ over that arc and, therefore, $h$ maps the arc into a point. Yet $h$, a harmonic polynomial of degree $n$, has finite valence $\leq n^2$, a contradiction.

If $f'\neq 0$ on $\partial D$ then all the functions (i.e. $f(z)$, $v(z)$ and $q'_m(z)$) that appear in the expression \eqref{eq:V-cusp} for the tangent vector, $V$, of the caustic are smooth along $\partial D$. Therefore, from \eqref{eq:V-cusp}, a singularity of $h(\partial D)$ occurs only when $V$ changes sign, hence $h(\partial D)$ must have a cusp there.  The lemma follows immediately.

If $f'(z_0)=0$ for $z_0\in\partial D$ (i.e., it is  a critical point of the lemniscate), then $|v(z_0)|=\infty$ by \eqref{eq:v}. However, the argument  of $V$ (i.e., $\arg V$) is still either continuous or jumps by $\pi$.  And a cusp occurs when $V$ changes direction, i.e., $\arg V$ jumps by $\pi$.
\end{proof}

If $f'(z_0)=0$ for $z_0\in\partial D$, multiple components of $\Omega$ merge together at $z_0$.  It turns out that the image of $\partial \Omega$ under $h$ has an interesting structure, as we will see below.

\vspace{0.3cm}
\noindent{\it Remark (a critical point on the critical lemniscate).}
Let $z_0\in\partial\Omega$ satisfy $f'(z_0)=\cdots=f^{(k-1)}(z_0)=0$ and $f^{(k)}(z_0)\neq 0$, i.e.,
\begin{equation*}
 	f(z)=f(z_0) + \frac{f^{(k)}(z_0)}{k!}(z-z_0)^{k}+{\mathcal O}\big((z-z_0)^{k+1}\big).
 \end{equation*}
Taking the absolute value and rotating if necessary, we obtain
\begin{equation}\label{eq:fabs}
 	|f(z)|=|f(z_0)| + \frac{|f(z_0)|}{k!}{\rm Re}\left[\frac{f^{(k)}(z_0)}{f(z_0)}(z-z_0)^{k}\right]+{\mathcal O}\big((z-z_0)^{k+1}\big).
 \end{equation}
Locally, the lemniscate consists of $2k$ curved rays meeting at $z_0$ and it divides the plane into $2k$ wedge-shaped sections with the angle $\pi/k$ at $z_0$.  Among them, total $k$ sections, equally spaced, are included in $\Omega$.
To be more precise, taking a sufficiently small disk $B$ centered at $z_0$, $B\cap\Omega$ has exactly $k$ components such that $B\cap\partial\Omega$ consists of $2k$ curves emanating from $z_0$ with the angular directions given by
\begin{equation*}
\theta_j=	\frac{1}{k}\arg\frac{f(z_0)}{f^{(k)}(z_0)}+\frac{\pi}{k}\bigg(j+\frac{1}{2}\bigg),\quad j=0,1,\cdots,2k-1.
\end{equation*}
Note that the sections between $\theta_{2\ell}$ and $\theta_{2\ell+1}$ for every $\ell=0,1,\cdots, k-1$ are in $\Omega$.  Therefore, taking the single component (let us denote it by $D_\ell$) of $B\cap\Omega$ between $\theta_{2\ell}$ and $\theta_{2\ell+1}$, the tangent vector of its boundary changes angular direction from $\theta_{2\ell+1}-\pi$ to $\theta_{2\ell}$ at $z_0$.    Using \eqref{eq:V-cusp}, assuming $z_0$ is not in the branch cut of $\sqrt{f}$, the corresponding tangent vector of $h(\partial D_\ell)$ changes the direction by $\pi$ at $h(z_0)$ (i.e., has a cusp singularity) if and only if
\begin{equation*}
	{\rm Im}\left(e^{i(\theta_{2\ell+1}-\pi)}q'_m(z_0)\sqrt{f(z_0)}\right) 	{\rm Im}\left(e^{i\theta_{2\ell}}q'_m(z_0)\sqrt{f(z_0)}\right)<0
\end{equation*}
or, equivalently,
\begin{equation*}
	\theta_{2\ell}+\arg\left(q'_m(z_0)\sqrt{f(z_0)}\right)\in \left(0,\frac{k-1}{k}\pi\right)\cup\left(\pi, \frac{2k-1}{k}\pi\right).
\end{equation*}
When $k$ is even (respectively odd) there can be at most two (respectively one) values of $\ell$'s that do {\em not} satisfy the above condition.  It means that for those values of $\ell$'s $h(\partial D_\ell)$ has a smooth boundary at $h(z_0)$, and for the other values of $\ell$'s, $h(\partial D_\ell)$ has a cusp at $h(z_0)$.  In Figure \ref{fig:petals}, the middle picture shows a caustic where the critical point $z_0=0$ corresponds to the three cusps at the origin, and the last picture shows a caustic where one component (red) of the lemniscate maps $z_0$ to a regular point of the caustic.

\begin{figure}
\includegraphics[width=0.2\textwidth]{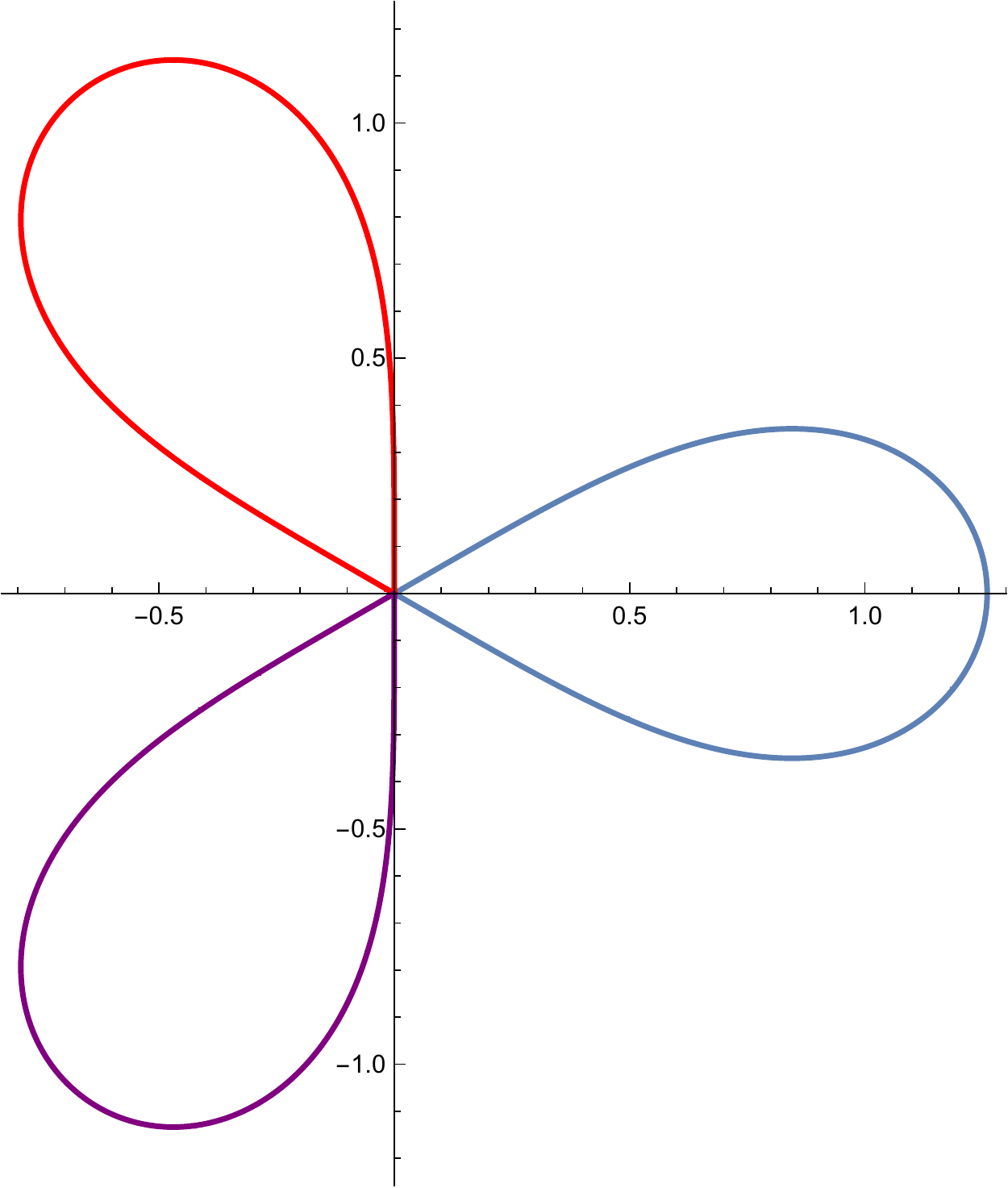}
\includegraphics[width=0.35\textwidth]{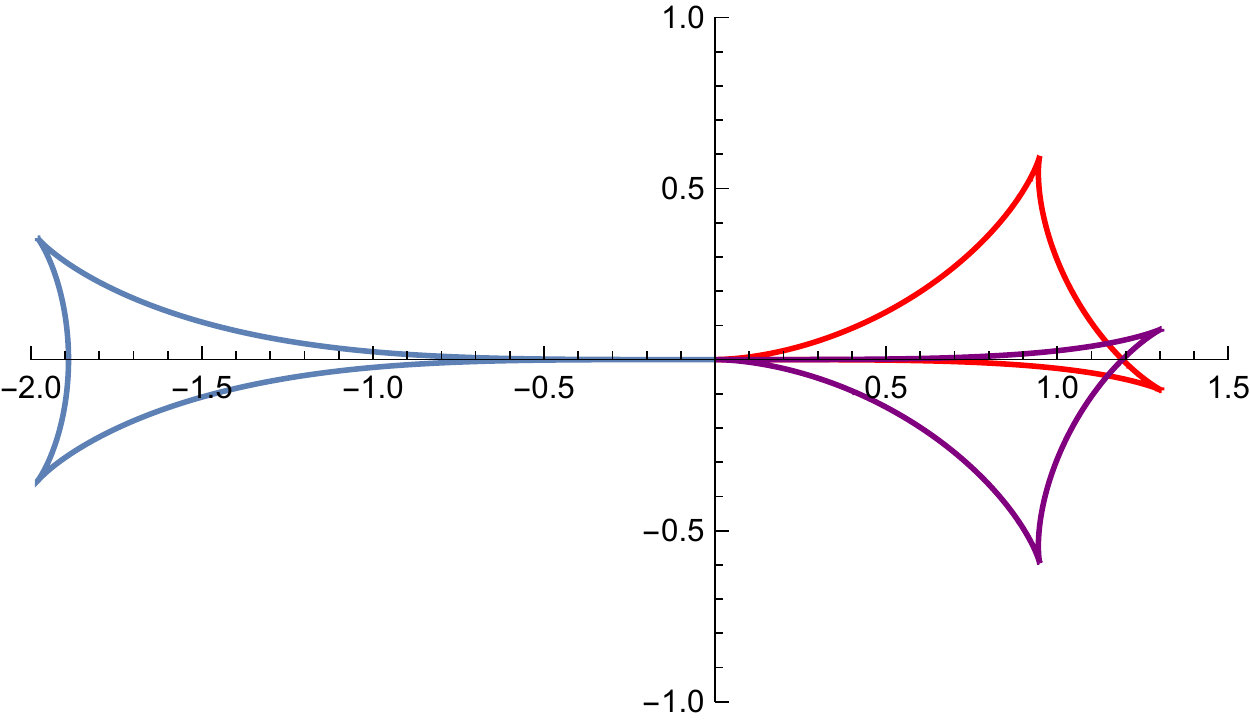}
\includegraphics[width=0.35\textwidth]{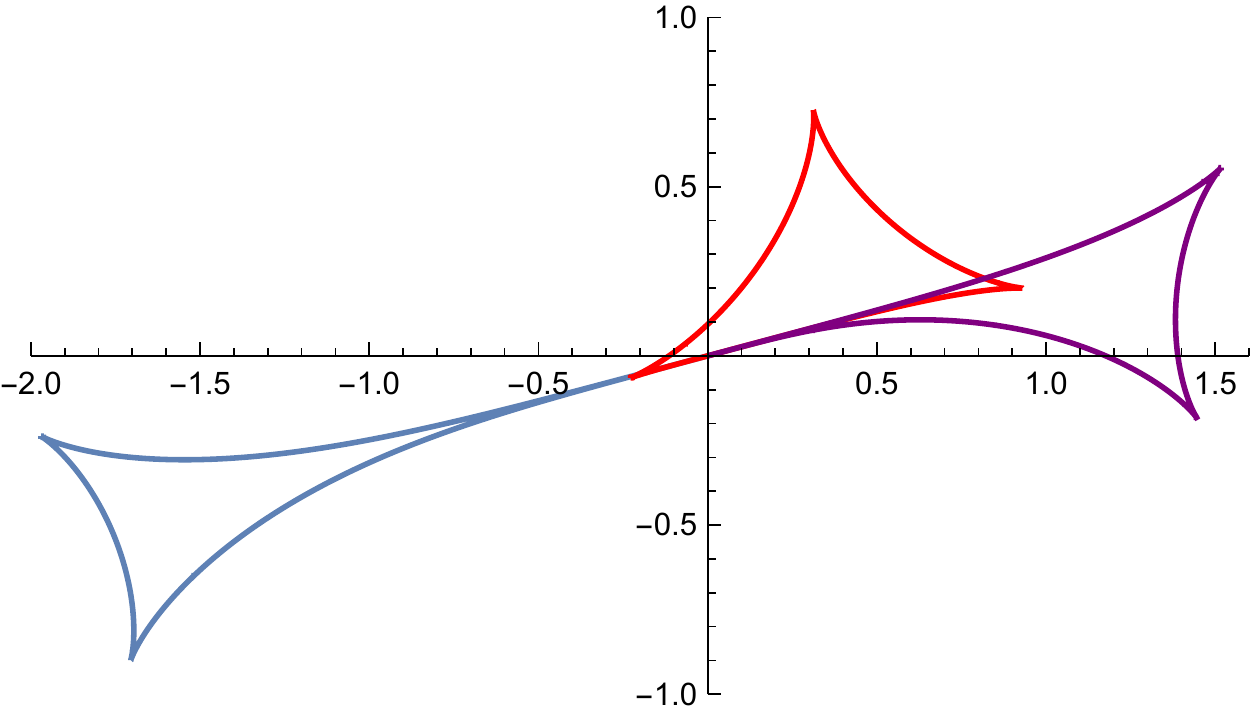}
\caption{\label{fig:petals} Lemniscate (left) and caustics, $h(z)=\frac{z^4}{4}-z-e^{i\theta}\overline z$ for $\theta=0$ and $\theta=\frac{\pi}{6}$.}
\end{figure}

\begin{lemma}\label{lem:cuspnumber} Let $D$ be a {\em simply connected} component of $\Omega$ with exactly $k$ zeros of $f$.  The number of cusps in $h(\partial D)$ is odd (resp. even) when $k$ is odd (resp. even) and, moreover, is $\geq 2+k$.
\end{lemma}

\begin{proof}
According to Lemma \ref{lem:caustic-cusp} a cusp in $h(\partial D)$ occurs whenever
${\rm Im}\big( v \,q'_m \sqrt f \big)$  changes sign.  To locate such events, it is convenient to define the function, $\Psi:[0,2\pi k)\to\RR$ by
\begin{equation}\label{eq:Psi}
\Psi:\theta\to \arg\Big( v(z) \,q'_m(z) \sqrt{f(z)} \Big)_{z=\eta(\theta)} = \arg v(\eta(\theta)) + \arg q'_m(\eta(\theta)) + \arg \sqrt{f(\eta(\theta))}.
\end{equation}
For the term $\arg q'_m(\eta(\theta))$, we choose the branch of the function ``$\arg$'' such that the term is continuous with respect to $\theta$.
For the term $\sqrt {f(\eta(\theta))}$ we  choose the branch of $\sqrt f$ and ``$\arg$'' function (separately from the previous one), so that that term is continuous.
Lastly, for the first term, $\arg v$, we choose the branch of ``$\arg$'' such that the term is a piecewise continuous function where the only discontinuities are at the critical points of the lemniscate (i.e., where $f'(\eta(\theta))=0$).  At the discontinuity $\arg v$ jumps by a positive angle, $\frac{k-1}{k\pi}\in [\pi/2,\pi)$ from $\theta_{2\ell+1}-\pi$ to $\theta_{2\ell}$ using the notations in the above Remark  ``a critical point on the critical lemniscate''; the angle can vary according to the order $k\geq 2$ of the critical point.  As a consequence, $\Phi$ is a continuous function with only discontinuities being the jump(s) by angles in $[\pi/2,\pi]$ at the critical points of the lemniscate.

For a cusp to occur at $\theta$, the condition in Lemma \ref{lem:caustic-cusp} gives that
\begin{equation}\label{eq:graph}
\bigcap_{\epsilon>0}(\Psi(\theta-\epsilon),\Psi(\theta+\epsilon)) \text{~ contains $ l\pi$ for some integer $l$.}
\end{equation}
Let us look at the three terms in $\Psi$ individually. The last term is linear in $\theta$ with the slope $1/2$ by \eqref{eq:V-sqrt-f} and, therefore, $\Delta \arg \sqrt {f} = k\pi$.  (Here and below $\Delta$ stands for the increment over $\partial D$.) The second term, $\arg q'_m$, is a continuous function and $\Delta \arg q'_m=0$ because $q'_m$ does not vanish in the closure of $D$.  Lastly, the first term, $\arg v$, is a piecewise continuous function with $\Delta \arg v = 2\pi$ where the only discontinuities are at the critical points of the lemniscate (i.e., where $f'(\eta(\theta))=0$).    Summing up, the total increment of $\Psi$ over $[0,2k\pi)$ is $(2+k)\pi$ and, therefore, there are at least $k+2$ points (and exactly $k+2$ if $\Psi$ is monotone) where the condition \eqref{eq:graph} holds.

Suppose now that $\Psi$ is not monotone.  For any point $\theta$ that satisfies the condition \eqref{eq:graph} there are two possibilities: as $\epsilon\to +0$,
 $$\text{(A):}~~ \Psi(\theta-\epsilon)<l\pi \text{ and } \Psi(\theta+\epsilon)>l\pi\quad\text{ or }\quad  \text{(B):}~~\Psi(\theta-\epsilon)>l\pi \text{ and } \Psi(\theta+\epsilon)<l\pi. $$
Since the total increment of $\Psi$ is $(k+2)\pi$, the number of points satisfying the condition (A) must be larger than those satisfying (B) by exactly $k+2$.
Therefore the number of points satisfying either (A) or (B) can be bigger than $k+2$ by an even number.  As a consequence, the number of cusps is always odd (respectively even) when $k$ is odd (respectively even).
\end{proof}

\begin{lemma}\label{lem:cusps-univalency}
Let $D$ be a component of $\Omega$ with a single zero of $f$. If $h(\partial D)$ has only three cusps, then it is a Jordan curve. There are more than three cusps in $h(\partial D)$ (hence five or more according to Lemma \ref{lem:cuspnumber}) if and only if there exists a point with the winding number of $h(\partial D)$ bigger than one, i.e., there exists $p\in\CC$ such that $\Delta_{\partial D}\arg (h(\cdot)-p)\geq 2\pi$.
\end{lemma}

\begin{proof}
To prove the first statement of the lemma, it suffices to show that the three smooth (open) arcs between the three cusps do not intersect each other.
Choosing any two arcs, there exists a cusp where the two arcs meet.
Since the tangent vector rotates at most by $\pi$ over a smooth part of the curve (cf. Lemma \ref{lem-caustic}), the two arcs can only get farther from each other as one moves along the arcs starting from the common cusp.
(Note however that, if the constant curvature is bigger than $1/2$, $h(\partial D)$ can self-intersect.) 

To prove the second statement, assume $h(\partial D)$ has five or more cusps.  Let $\Gamma$ be a smooth curve that is obtained by slightly ``smoothing'' all the cusps of $h(\partial D)$, see the left picture in Figure \ref{fig:smoothcusp}.  Considering infinitessimally small smoothing, such deformation indicates that the tangent vector of $\Gamma$ rotates by
 $$\frac{1}{2}\cdot 2\pi-\#\{\text{cusps}\}\pi\leq \pi -5\pi = - 4\pi$$
 over the whole curve.
\begin{figure}
\includegraphics[width=0.6\textwidth]{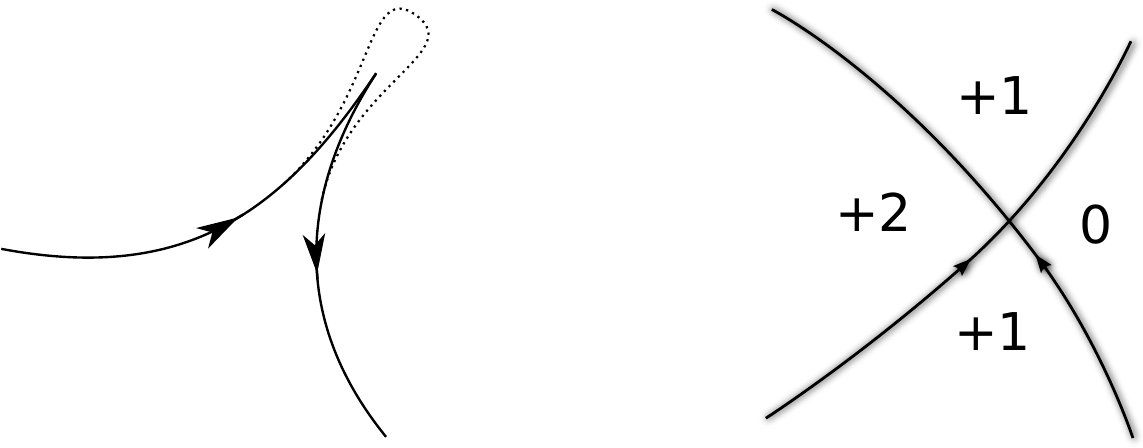}
\caption{\label{fig:smoothcusp} Cusp in caustic (left). Dotted line shows a ``cusp after smoothing''.  Relative winding numbers in a region around an intersection point (right).}
\end{figure}
Therefore $\Gamma$ cannot be a (smooth) Jordan curve and $h(\partial D)$ must have a self-intersection (that cannot be removed by a small perturbation, such as the one in the right picture of Figure \ref{fig:smoothcusp}).

Let us define the orientation on $h(\partial D)$ by the orientation inherited from $\partial D$.  For each point $p\notin h(\partial D)$, one can consider the winding number of $h(\partial D)$ around $p$.  The winding number of $h(\partial D)$ on the left side (with respect to the orientation) is bigger than the one on the right side by $+1$.  Then, in a neighborhood of a self-intersection, there must be a pair of regions where the winding numbers differ by $2$, see the right picture in Figure \ref{fig:smoothcusp}.
\end{proof}

Now we prove Theorem \ref{thm-3}.

To have more than three cusps, the function $\Psi$ that is defined in the proof of Lemma \ref{lem:cuspnumber} must be non-monotonic, i.e., there must be a point where the slope of the graph of $\Psi$ is negative, i.e.
\begin{equation}\label{eq:negative-slope}
\frac{d\Psi(\theta)}{d\theta}=	\frac{d\arg v(\eta(\theta))}{d\theta} + \frac{d\arg q'_m(\eta(\theta))}{d\theta} + \frac{1}{2}<0.
\end{equation}
Note that the first two terms in the left hand side give the curvature of $q_m(\partial D)$ with respect to the parametrization by $\theta$ or, equivalently, $\kappa/|f'|$ where $\kappa$ is the curvature with respect to the arclength on $\partial D$ (cf. \eqref{eq:v}):
$$ \left|\frac{d\eta(\theta)}{d\theta}\right| \kappa= \frac{d}{d\theta}\arg\frac{d q_m(\eta(\theta))}{d\theta} =\frac{d\arg v(\eta(\theta))}{d\theta} + \frac{d\arg q'_m(\eta(\theta))}{d\theta}.  $$
 So the inequality \eqref{eq:negative-slope} is exactly the one that appears in i) of Theorem \ref{thm-3}.

Once the graph of $\Psi$ is non-monotonic, one can create more roots of $\Psi \equiv 0 \mod\pi$ by vertically shifting this graph. This is done by the following transformation, which preserves the lemniscate $\{z:|p'_n(z)/q'_m(z)|=1\}$:
  $$p_n \to e^{i\varphi} p_n, \quad \varphi\in{\mathbb R}.$$
Under this transformation, $$\arg(v\,q'_m\sqrt f) \rightarrow \arg(v\,q'_m\sqrt f)+\varphi/2$$ because of the term $\sqrt f$.  This means that, for some $\varphi$, $h(\partial D)$ can have more than five cusps.   By Lemma \ref{lem:cusps-univalency}, there exists a point, say $p\in {\mathbb C}$, where $\Delta_{\partial D}(h(\cdot)-p)>2\pi$.  This means that
  $$\widetilde p_n(z) =  e^{i\varphi} p_n(z) - p $$
satisfies $\Delta_{\partial D}\widetilde h>2\pi$ where $\widetilde h = \widetilde p_n +\overline{q_m}$ and, therefore, $\widetilde h$ has at least two roots inside $D$.  This ends the proof of Theorem \ref{thm-3}.

\section{Construction of non-convex lemniscate}\label{sec:nonconvex}

Here we explain how we found the example in Figure \ref{fig:41concave}.

\begin{thm}\label{lem-no-inflec}
Let $f$ be a rational function and $\Omega=\{z:|f(z)|<1\}$.
Let $z_0 \in \p \Omega$, $f'(z_0) = 0$  and $(\log f)''(z_0) \neq 0$.
In a neighborhood of $z_0$, $\partial\Omega$ is a union of two smooth arcs that intersect perpendicularly at $z_0$.  Moreover, the point $z_0$ is not an inflection point (i.e. the curvature is strictly positive or negative at $z_0$) for either of the arcs if and only if
\begin{equation*}
	{\rm Re}\bigg( e^{\pm i\pi/4} \frac{(\log f)'''(z_0)}{(\log f)''(z_0)^{3/2}}\bigg)\neq 0.
\end{equation*}
\end{thm}

\vspace{0.3cm}
\begin{cor}\label{cor:concave} If $z_0\in\partial\Omega$ satisfies the assumptions in Theorem \ref{lem-no-inflec} and, furthermore, is not an inflection point of $\partial\Omega$, then there is a connected component of $\Omega$ whose curvature of the boundary (with respect to the counterclockwise orientation) converges to a negative value at $z_0$. 
\end{cor}
\begin{proof}
According to Theorem \ref{lem-no-inflec} there exists an open neighborhood $U$ of $z_0$
  such that $\Omega\cap U$ is the disjoint union of two domains, see Figure \ref{fig:two-cases} for an illustration ($\Omega\cap U$ is the shaded region).
When the ``no inflection'' condition in Theorem \ref{lem-no-inflec} is satisfied, the local configuration of $\Omega$ has two possibilities:
$\Omega\cap U$ is the disjoint union of a convex and a non-convex domain (left in Figure \ref{fig:two-cases}) or, $\Omega\cap U$ is the disjoint union of two non-convex domains (right in Figure \ref{fig:two-cases}).  In either case, there exists a connected component of $\Omega\cap U$ whose boundary has a negative curvature in a neighborhood of $z_0$.
\end{proof}

\begin{figure}
\begin{centering}
\includegraphics[width=0.55\textwidth]{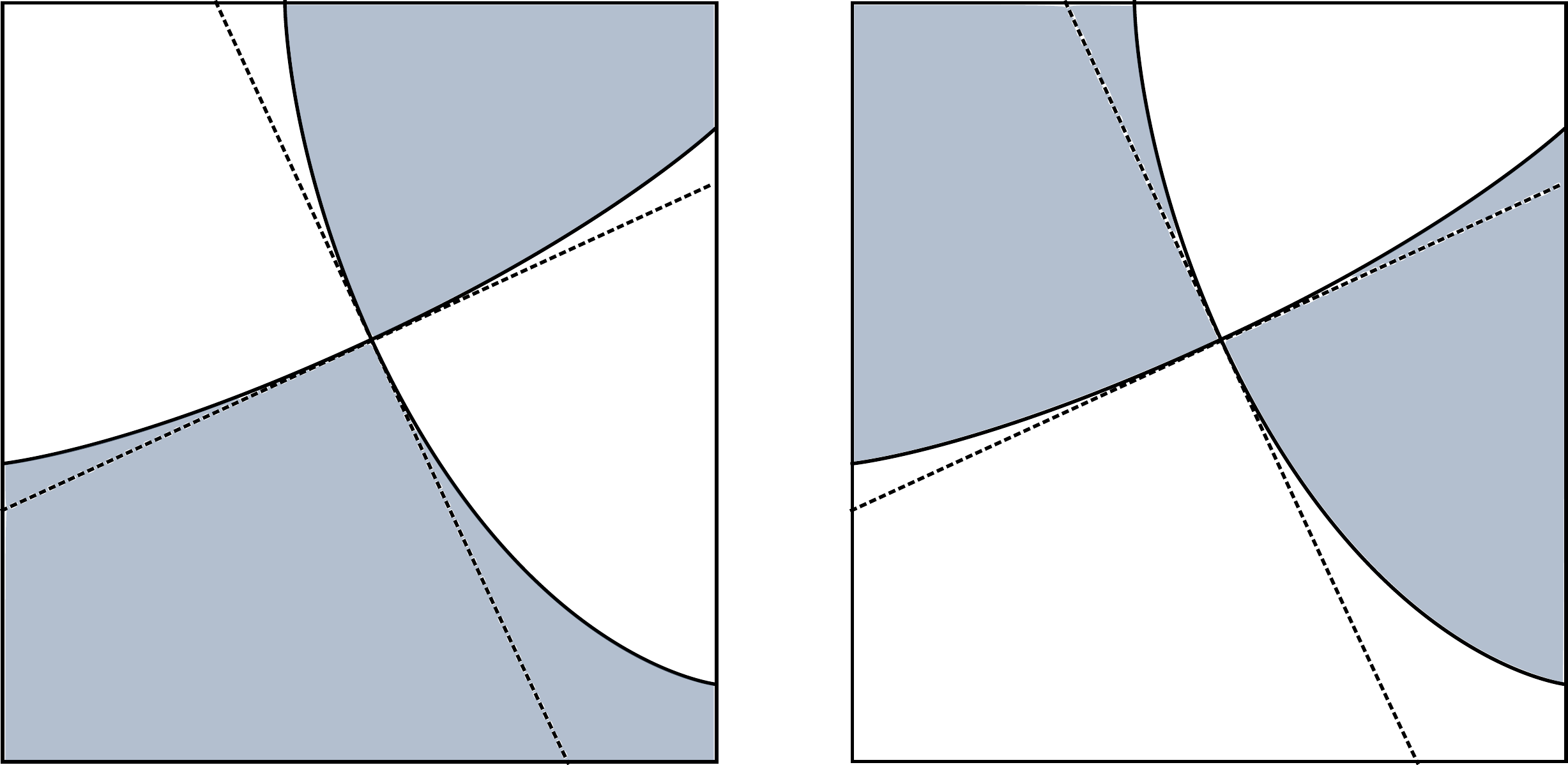}
\caption{\label{fig:two-cases} Lemniscate near $z_0$, see text below Theorem \ref{lem-no-inflec}.}
\end{centering}
\end{figure}

\noindent{\em Remark.}
{ To relate Theorem \ref{lem-no-inflec} to our harmonic polynomials, we take $f(z) = p_n'(z)/q_m'(z)$.
For $m=1$, as $z$ approaches $z_0$ along the ``concave'' arc, $\kappa(z)<0$ by Corollary \ref{cor:concave} and $|f'(z)|\to 0$.  Therefore, one obtains that $\kappa(z)/|f'(z)|\to -\infty$ as $z$ approaches $z_0$ along the ``concave'' arc.
Hence we obtain the ``non-convex'' domain $D$ as is needed to satisfy the condition $\kappa/|f'|<-1/2$ in Theorem \ref{thm-3}.
}

\vspace{0.3cm}
In the rest of this section, we prove Theorem \ref{lem-no-inflec}.

Let $f$ be a meromorphic function.
Let $\gamma:\RR\to\CC$ be the parametrization of a curve in $\CC$ such that $\gamma(0) = z_0$ and $f'(z_0) = 0$. Taking the derivatives of $\log \lvert f(\gamma(t)) \rvert$ at $t=0$ we obtain:
$$\frac{\textrm{d}}{\textrm{d}t} \log \lvert f(\gamma(t)) \rvert \bigg|_{t=0}=0,$$
$$\frac{\textrm{d}^2}{\textrm{d}t^2} \log \lvert f(\gamma(t)) \rvert\bigg|_{t=0} = \textrm{Re}\Big(\dot\gamma(0)^2 (\log f)''(z_0)\Big),$$
$$\frac{\textrm{d}^3}{\textrm{d}t^3}\log \lvert f(\gamma(t)) \rvert \bigg|_{t=0}= \textrm{Re}\Big(3\dot\gamma(0)\ddot\gamma(0) (\log f)''(z_0)+\dot\gamma(0)^3(\log f)'''(z_0)\Big).$$
Above, $\dot\gamma$ and $\ddot\gamma$ denote, respectively, the derivative and the second derivative of $\gamma$ with respect to $t$.
Consider the Taylor expansion of $\log|f(\gamma(t))|$ around $t=0$. We have
\begin{equation*}\begin{split}
\log|f(\gamma(t))| &= \textrm{Re}\Big(\dot\gamma(0)^2 (\log f)''(z_0)\Big)\frac{t^2}{2}
\\&+\textrm{Re}\Big(3\dot\gamma(0)\ddot\gamma(0) (\log f)''(z_0)+\dot\gamma(0)^3(\log f)'''(z_0)\Big)\frac{t^3}{6} + {\mathcal O}(t^4).
\end{split}
\end{equation*}
If $\gamma$ parmametrizes the lemniscate $\partial \Omega$ passing through $z_0$, where $\Omega$ is as in \eqref{eq:Omega}, then we have $\log|f(\gamma(t))|=0$ identically for all $t$, and all the coefficients in the above Taylor series must vanish.  The first coefficient vanishes when
\begin{equation}\label{eq:use1}
	\dot\gamma(0)^2 =  i c_1   \ol{(\log f)''(z_0)},\quad c_1\in\RR.
\end{equation}
We may assume that $\gamma$ is the arclength parametrization, i.e., $|\dot\gamma|\equiv 1$.
Then $\ddot\gamma\overline{\dot\gamma}$ is purely imaginary and
\begin{equation}\label{eq:use2}\ddot\gamma(0)=i c_2\, \dot\gamma(0)
\end{equation} for some real constant $c_2$.  The second coefficient in the Taylor series gets simplified as
\begin{equation*}
\begin{split}
&	\textrm{Re}\Big(3\dot\gamma(0)\ddot\gamma(0) (\log f)''(z_0)+\dot\gamma(0)^3(\log f)'''(z_0)\Big)
\\
&	=- 3c_1c_2 |(\log f)''(z_0)|^2 -c_1 {\rm Im}\Big( \dot\gamma(0)\overline{(\log f)''(z_0)} (\log f)'''(z_0)\Big),
	\end{split}
\end{equation*}
using \eqref{eq:use1} and \eqref{eq:use2}.
 The above  expression vanishing implies that
\begin{equation*}
	c_2 = - \frac{1}{3}{\rm Im}\bigg( \dot\gamma(0)\frac{(\log f)'''(z_0)}{(\log f)''(z_0)}\bigg).
\end{equation*}
Summarizing, we obtain
\begin{equation*}
	\gamma(t) = z_0 + \dot\gamma(0) t - \frac{i\,\dot\gamma(0)}{6}{\rm Im}\bigg( \dot\gamma(0)\frac{(\log f)'''(z_0)}{(\log f)''(z_0)}\bigg) t^2 + {\mathcal O}(t^3),
\end{equation*}
where, using \eqref{eq:use1}, 
\begin{equation*}
	\dot\gamma(0)=\pm e^{\pm i\pi/4}i  \frac{\overline{\sqrt{(\log f)''(z_0)}}}{|\sqrt{(\log f)''(z_0)}|}=\pm e^{\pm i\pi/4}i  \frac{|\sqrt{(\log f)''(z_0)}|}{\sqrt{(\log f)''(z_0)}}.
 \end{equation*}
The two signs can be chosen arbitrarily and independently.  This proves that there are two different arcs orthogonal at $z_0$. The curve $\partial\Omega$ does not have an inflection point at $z_0$ when the quadratic term in the Taylor expansion of $\gamma(t)$ is non-zero, i.e.,
\begin{equation*}
	{\rm Re}\bigg( e^{\pm i\pi/4} \frac{(\log f)'''(z_0)}{(\log f)''(z_0)^{3/2}}\bigg)\neq 0.
\end{equation*}
The proof is now complete.

\end{document}